\documentclass[a4paper,11pt]{article}
\usepackage{amssymb}
\usepackage{mathrsfs}
\usepackage{indentfirst}
\usepackage{amsmath}
\usepackage{amsfonts,amssymb}
\usepackage{bm}\usepackage{amsmath}
\usepackage{amssymb}
\usepackage{amsthm}
\usepackage{graphicx,color}
\catcode`\@=11 

\catcode`\@=12
\newtheorem{theorem}{Theorem}[section]

\newtheorem{remark}{Remark}[section]

\newtheorem{proposition}{Proposition}[section]

\vspace{-10mm}
\title{Wave Patterns in A Nonclassic Nonlinearly-Elastic Bar Under Riemann Data\footnote{This work was supported by a GRF grant from
Hong Kong Research Grants Council (Project No. CityU 11303015) and
the National Natural Science Foundation of China (Grant Nos.
11301005, 11301006,11572272) and Anhui Provincial Natural Science
Foundation (Grant No. 1408085MA01). K. R. Rajagopal thanks the
Office of Naval Research for support of this work.}}

\author{Shou-Jun Huang$^{a}$,\;\;K. R. Rajagopal$^b$,\;\;Hui-Hui Dai$^c$\footnote{Corresponding
author.
 Email address: mahhdai@cityu.edu.hk (Hui-Hui Dai)}\\
$^a$School of Mathematics and Computer Science, Anhui Normal
University,\\ Wuhu 241002, P.R.China\\
$^b$Department of Mechanical Engineering, Texas A \& M University,\\
College Station, TX 77840, USA\\
$^c$Department of Mathematics, City University of Hong Kong,\\
Kowloon, Hong Kong, P.R.China}\date{ }
\begin{document}
\maketitle\vspace{-10mm}
\begin{abstract}Recently there has been interest in studying a new class of elastic
materials, which is described by implicit constitutive relations.
Under some basic assumption for elasticity constants, the system of
governing equations of motion for this elastic material is strictly
hyperbolic but without the convexity property. In this paper, all
wave patterns  for the nonclassic nonlinearly elastic materials
under Riemann data are established  completely by separating the
phase plane into twelve disjoint regions and by using
 a nonnegative dissipation rate assumption and the maximally dissipative kinetics at any stress discontinuity.
 Depending on the initial data, a
variety of wave patterns can arise, and in particular there exist
composite waves composed of a rarefaction wave and a shock wave. The
solutions for a physically realizable case are presented in detail,
which may be used to test whether the material belongs to the class
of classical elastic bodies or the one wherein the stretch is
expressed as a function of the stress.


\end{abstract}
\newpage
\section{Introduction}
Until recently, models used to describe the elastic response of
bodies belonged to either the class of Cauchy elastic bodies or
Green elastic bodies. Recently, Rajagopal \cite{raja, raja1}
introduced a much larger class of elastic bodies that included
Cauchy elastic bodies and Green elastic bodies as a subset, if by
elastic response one refers to a response wherein the body is
incapable of dissipating energy, that is, inability to convert
working into thermal energy. Of particular reference to the current
work are bodies defined by implicit constitutive relations between
the stress and the deformation gradient, or the sub-class wherein
the strain in the body is a function of the stress. Such models are
relevant when one has a material wherein the body exhibits a
limiting strain or when the response between the strain and stress
becomes non-linear even for very small strains wherein the classical
models of elasticity reduce to the classical linearized elastic
model. When the elastic body exhibits limiting strain then one could
encounter the possibility that the stress cannot be expressed as a
function of the strain (see Rajagopal \cite{raja}). A detailed
mathematical treatment of such a response can be found in Bulicek et
al. \cite{buli}. With regard to the possibility of a non-linear
relationship between the strain and the stress, even when the
strains are very small, one needs but look at the response of alloys
such as Gum metal  (see Saito et al. \cite{saito}) and many other
Titanium Nickel based alloys (see Talling et al. \cite{Tall}, Withey
\cite{withey}, Zhang \cite{zhang}). The response of such alloys
cannot be described by the classical linearized elastic response but
can be described very well with the help of the new class of elastic
models wherein the linearized strain is a non-linear function of the
stress (see Rajagopal \cite{raja21}). Another very important class
of problems where the new class of models might prove to be very
useful is in predicting the state of strain in the neighborhood of
cracks and the tips of notches, etc. While the linearized theory of
elasticity predicts strains that blow up in the neighborhood of the
tip of a crack, contradicting the very precepts under which the
approximation is derived, the new class predicts results that are
physically meaningful in that the strains are bounded and never
exceed the limit of small strain that is supposed (see Rajagopal and
Walton \cite{raja4}, Kulvait, et al. \cite{kul}).

Nonlinear waves in elastic bars, within the traditional framework
that the stress is a function of the strain, have been studied in
various contexts. For example, recently Huang, Dai, Chen and Kong
\cite{huang1} showed that for certain nonlinearly elastic materials,
it is possible to generate a phenomenon in which a tensile wave can
catch the first transmitted compressive wave (so the former can be
undermined) in an initially stress-free two-material bar. Depending
on the interval of the initial impact, the wave catching-up
phenomena can happen in two wave patterns. Some asymptotic solutions
were also constructed. As a continuation of this work, Huang, Dai
and Kong \cite{huang} investigated the wave catching-up phenomenon
in a nonlinearly elastic prestressed two-material bar and the global
structure stability of nonlinear waves was also proved by the method
of characteristics and the theory of typical boundary problems. An
interesting study on impact-induced phase transformation in a shape
memory alloy rod was carried out by Chen and Lagoudas \cite{chenyi},
and notably they also found that composite waves with a rarefaction
wave and a shock wave can arise.

In this paper, we study the Riemann problem for a specific sub-class
of the new class of elastic bodies proposed by Rajagopal and focus
on the various wave patterns.  These equations do not possess
convexity though they are strictly hyperbolic. In this study, the
Reimann problem for this special sub-class is solved completely. We
find that, depending on the initial condition, a variety of wave
patterns can arise including a composite wave comprising of a
rarefaction wave and a shock wave. We also note that due to the
implicit constitutive relation (\ref{16}), it is natural to select
the velocity and the stress as the unknowns. Within such a
framework, the equations of motion governing the sub-class of bodies
under consideration cannot be written in terms of the type of
conservation laws that hold for the classical elastic body.

To introduce the kind of constitutive relation adopted in this
paper, we first recall some basic definitions in kinematics. The
reference configuration, denoted by $\mathcal {B}$, is assumed to be
stress-free. A particle ${\bf X}\in\mathcal{B}$ occupies the
position ${\bf x}\in\mathcal{B}_t$, where $\mathcal{B}_t$ is the
configuration at time $t$, that is referred to as the current
configuration. The mapping that maps the reference configuration to
the current configuration is assumed to be one to one, and is given
by ${\bf x}={\bf\chi}({\bf X},t)$. We denote the displacement by
${\bf u}={\bf x}-{\bf X}$. Then the gradients of displacement are
given as
\begin{equation*}\label{3}\frac{\partial{\bf u}}{\partial{\bf X}}=\nabla_{\bf X}{\bf u}={\bf F}-{\bf
I}\quad \text{or}\quad \frac{\partial{\bf u}}{\partial{\bf
x}}=\nabla_{\bf x}{\bf u}={\bf I}-{\bf F}^{-1},\end{equation*} where
${\bf F}=\frac{\partial{\bf x}}{\partial{\bf X}}$ is the deformation
gradient tensor, and ${\bf I}$ is the identity tensor. The
Green-Saint Venant strain ${\bf E}$ is given by
\begin{equation}\label{8}{\bf E}=\frac12\left(\nabla_{\bf X}{\bf u}+(\nabla_{\bf X}{\bf u})^T+(\nabla_{\bf X}{\bf u})^T\nabla_{\bf X}{\bf u}\right).\end{equation}

When one assumes that the displacement gradient is small so that the
last term that appears in the right hand side of (\ref{8}) can be
ignored in comparison to the other terms, one obtains the linearized
measure of strain. The constitutive relation for elastic response
within the classical theory of Cauchy or Green elasticity then leads
to the popular approximation of linearized elasticity.  Recently,
Rajagopal \cite{raja} (see also Rajagopal \cite{raja1},
\cite{raja2}, \cite{raja21}) introduced the following implicit
constitutive relation for isotropic elastic materials
\begin{equation}\label{9}{\bf f}({\bf T}, {\bf B})={\bf 0},\end{equation}
where ${\bf B}={\bf F}{\bf F^T}$ is the left Cauchy-Green strain
tensor and ${\bf T}$ is the Cauchy stress tensor. The general class
(\ref{9}) includes Cauchy elastic bodies as a special sub-class and
another special subclass that is useful and is given by
\begin{equation}\label{10}{\bf B}=\tilde{\alpha}_0{\bf I}+\tilde{\alpha}_1{\bf T}+\tilde{\alpha}_2{\bf T}^2,\end{equation}
where the materials moduli  $\tilde{\alpha}_i (i=1,2,3)$ depend on
the density and the principal invariants of the Cauchy stress.
 Under the small strain
assumption
\begin{equation*}\label{11}\max_{{\bf X}\in\mathcal{B},t\in\mathbb{R}}||\nabla_{\bf X}{\bf u}||=O(\delta),\quad \delta\ll1,\end{equation*}
where $||\cdot||$ denotes the trace norm, Rajagopal \cite{raja}
obtained the approximation with $O(\delta)$ from (\ref{10}) as
follows
\begin{equation*}\label{12}{\bf \epsilon}=\alpha_0{\bf I}+\alpha_1{\bf T}+\alpha_2{\bf T}^2,\end{equation*}
where as usual the materials moduli  $\alpha_i (i=1,2,3)$ depend on
the density in current configuration and the principal invariants of
Cauchy stress, $\epsilon$ is the linearized strain tensor. In
particular, Kannan,  Rajagopal and Saccomandi  \cite{raja3} proposed
the following special constitutive relation:
\begin{equation}\label{13}{\bf \epsilon}=\beta(\text{tr}{\bf T}){\bf I}+\alpha\left(1+\frac{\gamma}{2}\text{tr}{\bf T^2}\right)^n{\bf T},\end{equation}
where $\alpha\geq0, \beta\leq0, \gamma\geq0$ and $n$ are constants.

There have been many studies carried out within the context of the
new class of elastic bodies defined by (\ref{13}). Of relevance to
the current study is the paper by Kannan,  Rajagopal and Saccomandi
\cite{raja3}, wherein they investigated the unsteady motions  of
this new class of elastic solids. It was shown that the stress wave
changes its shape since the wave speed depends on the stress and the
value of stress varies according to the thickness of the slab. All
these phenomena for the generated stress wave are quite different
from what one observes for a classical linear elastic material.


When we restrict the constitutive relation (\ref{13}) to one
dimension, we obtain the one-dimensional constitutive relation
\begin{equation}\label{16}\epsilon=\beta T+\alpha\left(1+\frac{\gamma}{2}T^2\right)^nT.\end{equation}
We will assume that the constants in (\ref{16}) satisfy that
\begin{equation}\label{17}\alpha>0,\quad \beta<0,\quad \gamma>0,\quad n>0.\end{equation}
Moreover, we suppose that
\begin{equation}\label{18}\alpha+\beta>0.\end{equation}
\begin{remark}The assumption (\ref{18}) guarantees the following governing system of equations (\ref{14})  is  hyperbolic.\end{remark}

In this paper, we consider the Riemann problem for nonlinear wave
equations
\begin{equation}\label{14}\rho\frac{\partial v}{\partial t}=\frac{\partial T}{\partial x},\quad \frac{\partial \epsilon}{\partial t}
=\frac{\partial v}{\partial x}\end{equation} with the initial data
\begin{equation}\label{15}(T,v)(0,x)=\left\{\begin{array}{ll}(T_l,v_l),&x<0,\vspace{2mm}\\(T_r,v_r),&x>0,\end{array}\right.\end{equation}
where $t, x$ represent the time and spatial coordinate respectively,
$\rho$ the density of elastic body, $T$ the Cauchy stress,
$\epsilon$ the strain, $v$ the particle velocity. The constant
Riemann data in (\ref{15}) satisfy that $(T_l,v_l)\neq(T_r,v_r).$

Riemann problem for PDEs is of significance not only in  physics,
but also in mathematics. It is well-known that the Riemann problem
can be used as a building block to prove existence results for the
Cauchy problem for (\ref{14}) with general initial data
\cite{glimm}, possibly having large total variation \cite{bressan}.

 For the gas dynamics equations with convex condition,
the Riemann problem has been well-studied (see \cite{ch},
\cite{smoller}). Wendroff \cite{wen, wen1} investigated the gas
dynamics equations without convexity conditions for the pressure and
constructed a solution to the Riemann problem.  Liu \cite{liu, liu1}
considered the Riemann problem for general systems of conservation
laws. By introducing an extended entropy condition, which is
equivalent to the Lax's shock inequalities \cite{lax} when the
system is genuinely nonlinear, Liu \cite{liu1} proved the uniqueness
theorem for the Riemann problem of the gas dynamics equations
without convexity conditions for the pressure. By a special
vanishing viscosity method, Dafermos \cite{dafer} obtained the
structure of solutions of the Riemann problem for a general
$2\times2$ conservation laws.
 Matsumura and Mei \cite{mei}
considered the nonlinear asymptotic stability of viscous shock
profile for a one-dimensional  system of viscoelasticity, where the
constitutive relation is non-convex. They applied the degenerate
shock condition proposed by Nishihara \cite{nish} to single out an
admissible shock solution. By introducing a generalized shock in
\cite{nish}, Sun and Sheng \cite{sun} constructed the solutions to
the Riemann problem for a system of nonlinear degenerate wave
equations in elasticity, for which the strain-stress function is
nonconvex. For the same equations, by using the Liu-entropy
condition in \cite{liu1} alternatively, Liu and Wang \cite{liux}
completely obtained the corresponding Riemann solutions, some of
which are different from those in \cite{sun}.

By Liu-entropy condition, LeFloch and Thanh \cite{lefloch} uniquely
solved the Riemann problem for a nonlinear hyperbolic system
describing phase transitions in elastodynamics. But it is noted that
the elastic model in \cite{lefloch} is different from the present
material by comparing the assumptions (1.3) in \cite{lefloch} and
(\ref{16}). A more related work was done by Tzavaras \cite{tza}, who
studied the Riemann problem for the equations of one-dimensional
isothermal elastic materials by taking viscosity to be zero in the
equations of viscoelasticity. Wendroff criterion was applied at
shocks to select a physically admissible one. For more recent
results on Riemann problems, one may refer to Chapter IX and
references therein by Dafermos \cite{dafer1} or a monograph by
LeFloch \cite{lefloch1}.

We note that while there are a number of analytical solutions
available for classical elastic materials, there are few ones for
the previously mentioned nonclassical materials. Considering the
importance of analytical solutions and the Riemann problem, we shall
solve (\ref{14}) and (\ref{15}) with the implicit constitutive
relation analytically. The Riemann problem (\ref{14})-(\ref{15}) is
different from the classical one since the linearized strain is a
function of stress and this function is nonconvex (cf. Remark 2.3).
From the application point of view, the mathematical results on
Riemann problem may be used to test whether the material belongs to
the class of classical elastic bodies or the ones with an implicit
constitutive relation by comparing the wave patterns in a designed
experiment. Remarkably, we find that there are twelve wave patterns
for the considered material, while there is only one wave pattern
for a classical one with the small strain.
 We remark that the well-posedness for the Cauchy
problem for one-dimensional strictly hyperbolic equations with small
initial data has been established by Bianchini and Bressan
\cite{bressan}. Error estimates for the Glimm approximate solution
and the vanishing viscosity solution were derived in \cite{fabio}.

The remaining organization of this paper is as follows: In Section
2, we briefly recall some admissibility criteria for weak solutions
of hyperbolic equations. Section 3 is devoted to constructing the
elementary waves for our system (\ref{14}) and in Section 4, we
provide all the solutions to the Riemann problem
(\ref{14})-(\ref{15}) case by case. Section 5 is devoted to studying
the Riemann solution in detail for a physically realizable situation
$v_l=v_r=0$.

\section{Admissibility criteria}
In the community of hyperbolic equations, one often uses the Lax
entropy inequality \cite{lax} to single out the unique weak solution
for genuinely nonlinear hyperbolic equations; while for general
hyperbolic equations without convexity, the Liu-entropy condition
(\cite{liu},\cite{liu1}) is a preferable candidate, which can be
viewed as a generalization of Oleinik entropy condition
\cite{dafer1} for scalar hyperbolic conservation laws. The
Liu-entropy condition reads
\begin{equation}\label{n7}s(U_l,U_r)\leq s(U_l,U)\;\text{for every}\;U\;\text{between}\;U_l\;\text{and}\;U_r,\end{equation}
where $s(U_l,U_r)$ is the speed of a shock connecting the left state
$U_l$ and the right state $U_r$. Both Lax entropy inequality and
Liu-entropy condition have been justified by the method of vanishing
viscosity (see  Chapter VIII in Dafermos \cite{dafer1}). It is also
worthy to point out that these two criteria can guarantee the shock
waves are stable \cite{dafer1}.

In the mechanics community, often a more physical selection
criterion is used.  Knowles \cite{knowles} investigated the
impact-induced tensile waves in a semi-infinite bar made of a
rubberlike material. The governing system of equations is strictly
hyperbolic, but genuine nonlinearity fails. He succeeded in
constructing the corresponding solutions according to three regimes
of response, depending on the intensity of the loading. For the
intermediate case, there is a one-parameter family of solutions to
the initial-boundary problem. In order to select the unique
admissible solution,  Knowles \cite{knowles} introduced the concept
of {\it driving force} defined via the dissipation rate (see also
related discussions in the monograph by Abeyaratne and Knowles
\cite{abe}) and the kinetic relations. By Eqs. (\ref{14}) and the
Rankine-Hugoniot conditions, the dissipation rate with shock waves
can be written as  (cf. \cite{abe}, \cite{knowles})
\begin{equation}\label{n1}D(t)=f(t)\,s,\end{equation}
where $f(t)$ is the {\it driving force} per unit cross-sectional
area acting at time $t$ and can be computed in terms of the stresses
on either side of the jump by
\begin{equation}\label{n2}f(t)=\int_{T_r}^{T_l}\varepsilon(y)dy+\frac{\varepsilon(T_r)+\varepsilon(T_l)}{2}(T_r-T_l),\end{equation}
$s=s(t)$ is the speed of a stress discontinuity. For the present
model (\ref{16}), we have
\begin{equation}\label{n3}f(t)=\frac{\alpha}{(n+1)\gamma}[F(T_l,T_r)-F(T_r,T_l)],\end{equation}
where $F(x,y)=(1+\frac{\gamma}{2}x^2)^n(1-\frac{n}{2}\gamma
x^2+\frac{n+1}{2}\gamma xy).$ By detailed but straightforward
analysis, the driving force $f(t)$ can be depicted as in Fig.
\ref{fig0}.
\begin{figure}[ht!]\begin{center}
\includegraphics[width=0.32\textwidth]{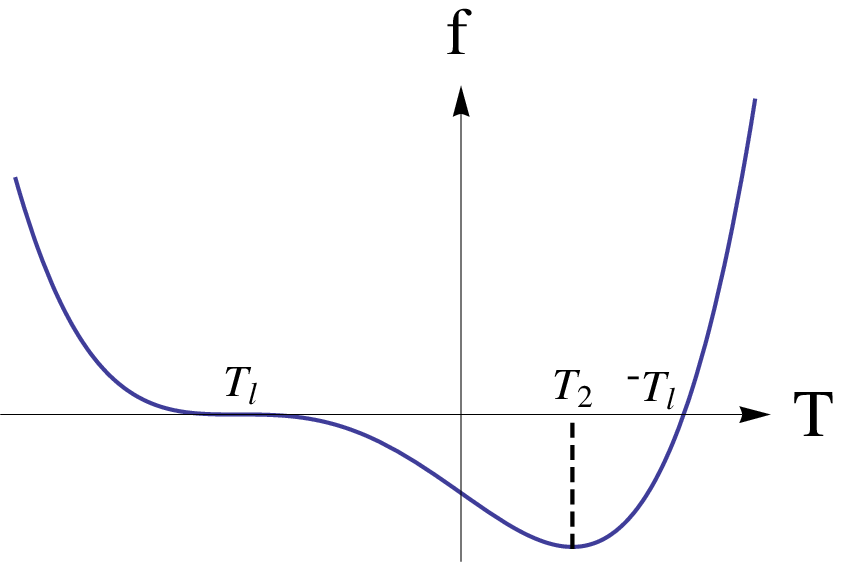}
\includegraphics[width=0.32\textwidth]{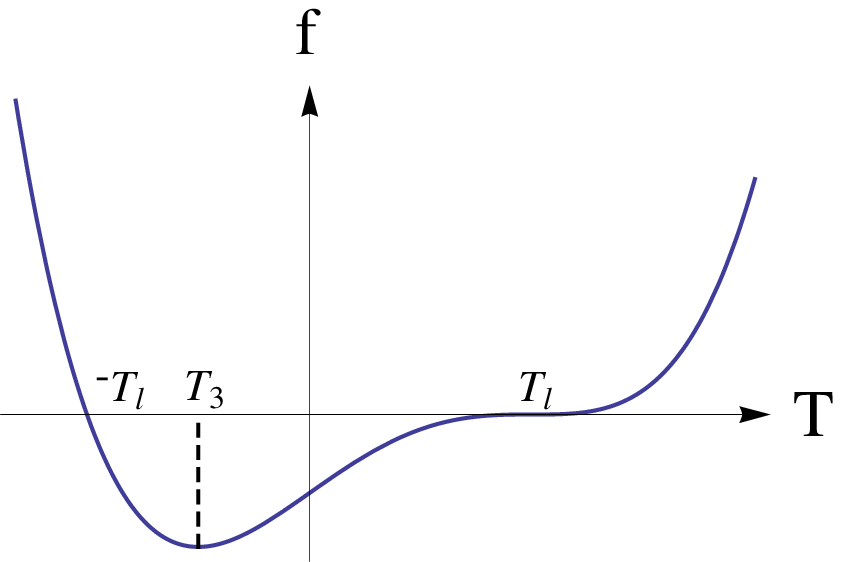}
\includegraphics[width=0.32\textwidth]{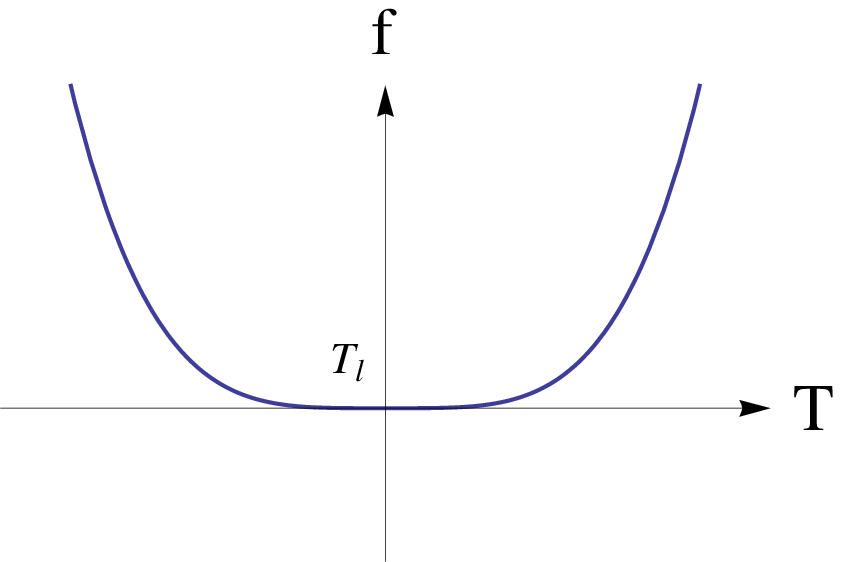}
  \renewcommand{\figurename}{Fig.} \caption{Plots for $f(t)$: $T_l<0$ (left), $T_l>0$(middle) and $T_l=0$ (right), where $T_2, T_3$ are defined through (\ref{38}). } \label{fig0}
 \end{center}
\end{figure}
The second law of thermodynamics requires $D(t)\geq0$. As in Knowles
\cite{knowles},  a solution to the Riemann problem (\ref{14}),
(\ref{15}) is called {\it physically admissible} if
\begin{equation}\label{n4}D(t)=f(t)\,s\geq0\;\;\text{for}\;\;t>0,\end{equation}
 equivalently,
\begin{equation}\label{n5}s\,(T_r-T_l)(T_r+T_l)\geq0\;\;\text{for}\;\;t>0,\end{equation}
at every stress discontinuity.  In order to obtained the uniqueness
of the solution to the corresponding initial boundary value problem,
Knowles \cite{knowles} considered a special kinetic relation:
maximally dissipative kinetics, which requires that

\begin{equation}\label{n6}\text{either}\;\;\frac{\sigma(\gamma^+)-\sigma(\gamma^-)}{\gamma^+-\gamma^-}=\sigma'(\gamma^+)
\;\;\text{or}\;\;\frac{\sigma(\gamma^+)-\sigma(\gamma^-)}{\gamma^+-\gamma^-}=\sigma'(\gamma^-)\end{equation}
must hold, where $\sigma=\sigma(\gamma)$ is the stress-response
function, $\gamma^{\pm}$ are the strains on either side of a
discontinuity.

In the next section, we shall discuss all backward and forward wave
curves according to the cases: $T_l<0$, $T_l>0$ and $T_l=0$. For the
case $T_l<0$, by the criterion (\ref{n5}) it is possible to
construct a backward shock wave for all $T_r\in(T_l,-T_l]$ (cf.
Fig.\ref{fig0}), while for $T_r>-T_l$, a single backward shock is
thermodynamically impossible since (\ref{n5}) cannot be satisfied.
Instead, there should appear an additional backward rarefaction
wave. That is to say, when $T_r>-T_l$, we have to establish a
two-wave solution composed of a rarefaction wave followed by a shock
wave. However, this kind of solution is not unique since the $T$ in
the middle state can lie arbitrarily in $(T_l,-T_l]$. With aiming to
overcome this difficulty, we apply the maximally dissipative
kinetics in \cite{knowles} at this solution, which requires that the
$T$ in the middle state should equal $T_2$ given by
\begin{equation}\label{38}\frac{\epsilon(T_2)-\epsilon(T_l)}{T_2-T_l}=\epsilon'(T_2) \;\;\text{for}\;\;
T_l<0\;\;
\text{and}\;\;\frac{\epsilon(T_3)-\epsilon(T_l)}{T_3-T_l}=\epsilon'(T_3)
\;\;\text{for}\;\; T_l>0,\end{equation} where  $T_2$ and $T_3$ are
uniquely determined due to the specific form of the strain-stress
relation (\ref{16}), see Fig.\ref{fig5}.
\begin{figure}[ht!]\begin{center}
\includegraphics[width=0.45\textwidth]{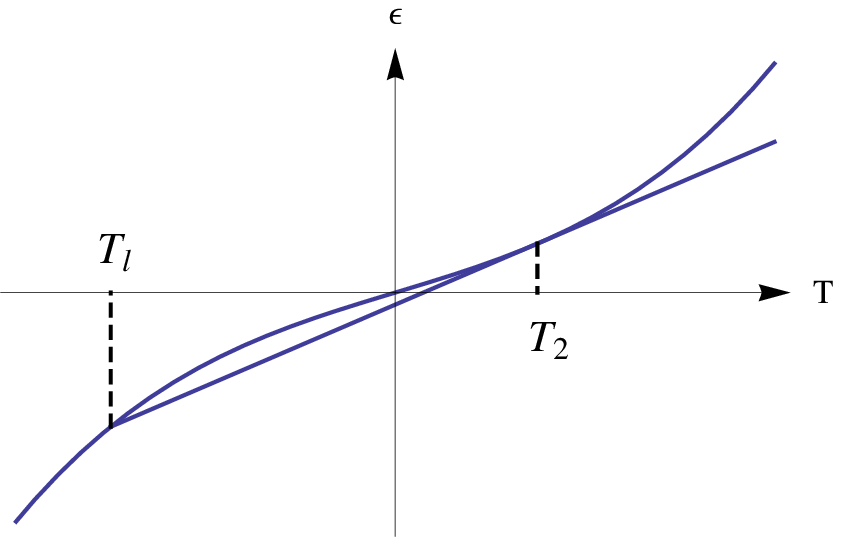}\hspace{2mm}\includegraphics[width=0.45\textwidth]{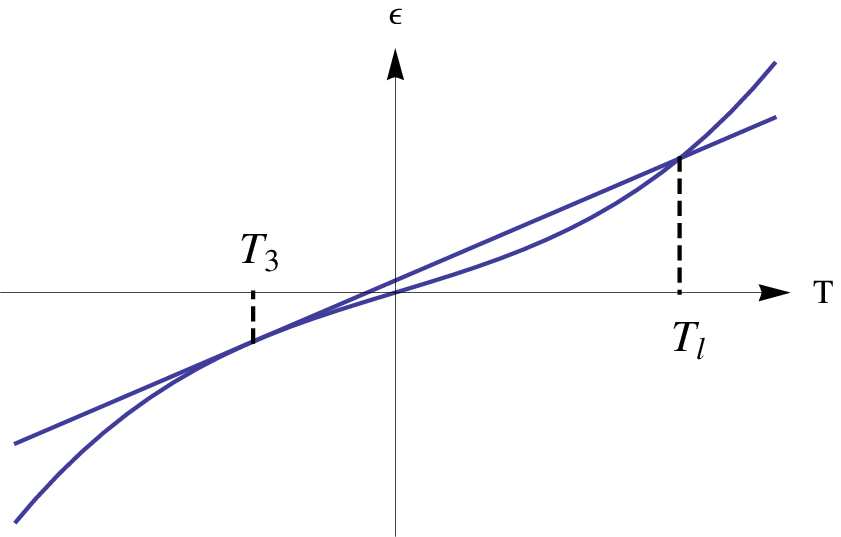}
 \renewcommand{\figurename}{Fig.} \caption{The strain-stress curve and a tangent ray.} \label{fig5}
\end{center}
\end{figure}
(\ref{38}) is the same as (\ref{n6}) in terms of the inverse
constitutive relation $\epsilon=\epsilon(T)$, which is the case for
the nonclassic elastic material (\ref{16}). Thus, by using the
maximally dissipative kinetics, we arrive at a unique two-wave
solution for $T_r>-T_l$. Actually, we have constructed such an
elementary wave curve including two parts, one shock wave curve from
$(T_l,v_l)$ to $(T_2,v_2)$ and the other rarefaction wave curve from
$(T_2,v_2)$ to $(T_r,v_r)$, see Fig.4.1. If we change to use a
stability criterion, such as Liu-entropy condition (\ref{n7}), we
can construct a backward shock wave only for $T_r\in(T_l,T_2]$ and
beyond $T_2$, there appears a unique two-wave solution, which
exactly coincides with the previous one. Interestingly, the shock
wave for $T_r\in(T_l,T_2]$ satisfies not only the Liu-entropy
condition, but also the Lax entropy inequality. When $T_r<T_l$, one
may obtain a forward shock wave according to the criterion
(\ref{n5}). Moreover, this shock satisfies the Lax entropy
inequality. The cases for $T_l>0$ and $T_l=0$ can be dealt with
similarly.  The above
 argument produces a reasonable observation that a physical discontinuity satisfied by the maximally dissipative kinetics must be stable.
This kind of relationship between Lax entropy inequality,
Liu-entropy condition and the maximally dissipative kinetics seems
not revealed in the literature.

\section{Elementary waves}By the method of wave curves, we next divide the discussions
into three cases: $T_l<0, T_l>0$ and $T_l=0$. For first two cases,
the phase plane is split into twelve disjoint regions and the
corresponding Riemann solutions are derived. It is worth pointing
out that there exist some composite wave solutions, which are
composed of a rarefaction wave and a degenerate shock wave.

Let
$U=\left(\hspace{-1mm}\begin{array}{c}T\\v\end{array}\hspace{-1mm}\right)$,
the system (\ref{14}) can be rewritten as
\begin{equation}\label{20}U_t+A(U)U_x=0,\end{equation}
where
\begin{equation*}A(U)=\left(\begin{array}{cc}0&-\frac{1}{\epsilon'(T)}\\-\frac{1}{\rho}&0
\end{array}\right).\end{equation*} Due to
the assumption (\ref{18}), it is easy to see that
\begin{equation}\label{21}\epsilon'(T)=\beta+\alpha\left(1+\frac{\gamma}{2}T^2\right)^{n-1}\left(1+\frac{1+2n}{2}\gamma T^2\right)>0.\end{equation}
By direct computation, the eigenvalues of $A(U)$ read
\begin{equation}\label{22}\lambda_1=-\frac{1}{\sqrt{\rho\epsilon'(T)}}<0<\lambda_2=\frac{1}{\sqrt{\rho\epsilon'(T)}}.\end{equation}
The right eigenvectors corresponding to $\lambda_{i} (i=1,2)$ can be
chosen as
\begin{equation}\label{23}r_1=\left(\hspace{-1mm}\begin{array}{c}-\rho\lambda_1\\1\end{array}\hspace{-1mm}\right),
\quad
r_2=\left(\hspace{-1mm}\begin{array}{c}-\rho\lambda_2\\1\end{array}\hspace{-1mm}\right),\end{equation}
respectively; while the left eigenvectors corresponding to
$\lambda_i (i=1,2)$ can be taken as
\begin{equation}\label{24}l_1=(1,-\rho\lambda_1),\quad l_2=(1,-\rho\lambda_2),\end{equation}
respectively.

 Summarizing the above argument leads to
\begin{proposition}Under the assumption (\ref{17})-(\ref{18}), the system (\ref{20}) is  strictly hyperbolic with two distinct eigenvalues
(see (\ref{22})), and the right (resp. left) eigenvectors can be
chosen as (\ref{23}) (resp. (\ref{24})).\end{proposition}

\begin{proposition}Under the assumption (\ref{17})-(\ref{18}), the characteristic fields $\lambda_{i} \,(i=1,2)$ for (\ref{20})
 are not genuinely nonlinear in
the sense of Lax \cite{lax}.
\end{proposition}
\begin{proof} It suffices to calculate the invariants
$\nabla\lambda_{i}\cdot r_{i}\,(i=1,2)$. By computation,
\begin{equation*}\label{25}\nabla\lambda_{i}\cdot r_{i}=\left(\frac{\partial\lambda_{i}}{\partial T},\frac{\partial\lambda_{i}}{\partial v}\right)
\cdot\left(-\rho\lambda_{i},1\right)=-\rho\lambda_{i}\frac{\partial\lambda_{i}}{\partial
T}=\frac{\epsilon''(T)}{2\left[\epsilon'(T)\right]^2},\end{equation*}
where
\begin{equation}\label{26}\epsilon''(T)=
\alpha n \gamma T\left(3+\frac{1+2n}{2}\gamma
T^2\right)\left(1+\frac{\gamma}{2}T^2\right)^{n-2}.\end{equation} So
the system (\ref{20}) is genuinely nonlinear in the sense of Lax if
$T\neq0$, however the genuinely nonlinearity is not valid when
$T=0$. Thus, the proof is completed.
\end{proof}

\begin{remark}By (\ref{21}) and (\ref{26}),  the strain-stress relation
 $\epsilon=\epsilon(T)$ is always increasing and has concave part and convex part on
$(-\infty,0]$ and  $(0,+\infty)$, respectively.
\end{remark}

Since the Riemann problem (\ref{14}) and (\ref{15}) are invariant
under stretching of coordinates: $(t,x)\rightarrow (c t, c x) \; (c$
is a constant), we seek the self-similar solution
$(T,v)(t,x)=(T,v)(\xi), \xi=x/t$. Then the Riemann problem
(\ref{14})-(\ref{15}) can be reduced into the following boundary
value problem
\begin{equation}\label{28}\left\{\begin{array}{lll}\rho\xi v_{\xi}+T_{\xi}=0,\\
v_{\xi}+\xi\epsilon'(T)T_{\xi}=0,\\
(T,v)(+\infty)=(T_{r},v_{r}),\\
(T,v)(-\infty)=(T_{l},v_{l}).\end{array}\right.\end{equation} We
know (\ref{28}) provides  either the constant state solution
$U=$Const, or the singular solution, i.e., the backward (or
1-)rarefaction wave and the forward (or 2-)rarefaction wave
corresponding to the eigenvalues $\lambda_1$ and $\lambda_2$,
respectively.

Moreover, the system (\ref{14}) admits discontinuous shock
solutions, which satisfy the Rankine-Hugoniot conditions at the
moving stress discontinuity located at $x=x(t)$
\begin{equation*}\label{19}s\rho[v]+[T]=0,\quad s[\epsilon]+[v]=0,\end{equation*}
 where  $[f]=f(t,x(t)+0)-f(t, x(t)-0)$ and $ s=dx(t)/dt$ is the speed of a shock wave. If
 we have $s=\lambda_i(U_l)$ or $s=\lambda_i(U_r)$, then the shock $x=x(t)$
 is called a degenerate shock. If   both these equalities  are
 fulfilled, then the discontinuity $x=x(t)$ is a so-called contact
 discontinuity.

Under the assumption that the strain-stress relation is convex, one
can completely describe the structure of shock waves and rarefaction
waves for the system of classical conservation laws (cf.
\cite{smoller}). However, for the nonconvex case and the system
considered here, the situation is much more complicated and there
appear multiple waves or composite waves, see the following
discussions.

Now we consider the elementary waves  for the Riemann problem
 (\ref{14}) and (\ref{15}). By definition here, an elementary wave is a single shock or a rarefaction wave, or a composite wave composed of
 more than one single shock or rarefaction wave. A single elementary wave
 can be a backward rarefaction wave, or a backward shock wave (denoted by $R_1$, $S_1$, respectively),
 or a forward rarefaction wave, or a forward shock wave (denoted by
 $R_2$, $S_2$, respectively). First, we note that
 \begin{equation}\label{30}\frac{\partial\lambda_1}{\partial T}=\frac{\epsilon''(T)}{2\epsilon'(T)\sqrt{\rho\epsilon'(T)}}\;\left\{
 \begin{array}{ll}\geq0,\;\;T\geq0,\\<0,\;\;T<0,\end{array}\right.
 \quad\frac{\partial\lambda_2}{\partial T}=-\frac{\epsilon''(T)}{2\epsilon'(T)\sqrt{\rho\epsilon'(T)}}\;\left\{
 \begin{array}{ll}<0,\;\;T>0,\\\geq0,\;\;T\leq0.\end{array}\right.\end{equation}
 By the given left
state $(T_l,v_l)$ in the Riemann data, we divide the discussions
into three cases.

 Case I \quad $T_l=0$.

The backward rarefaction wave is
\begin{equation}\label{31}R_1:\quad
v-v_l=\int_{0}^T\sqrt{\epsilon'(\tau)/\rho}\;d\tau,\;\; T>0 \;\;
\text{or}\;\; T<0,
 \end{equation}
where $T>0$ or $T<0$ is determined by the requirement of
$\lambda_1(U)>\lambda_1(U_l)$ for the rarefaction waves.

The forward shock wave is given by
\begin{equation}\label{32}S_2:\quad v-v_l=\left\{\begin{array}{ll}-\sqrt{T\epsilon(T)/\rho},&T>0,\\
\sqrt{T\epsilon(T)/\rho},&T<0.\end{array}\right.\end{equation}
 This
shock automatically satisfies (\ref{n5}) and
$\lambda_2(U_l)>s_2>\lambda_2(U_r),
s_2=1/\sqrt{\rho\,\epsilon(T)/T}.$

Case II \quad $T_l<0$.

First we consider the forward elementary waves. The forward
rarefaction wave $R_2$ can be constructed as
\begin{equation*}\label{34}R_2:\quad v-v_l=-\int_{T_l}^T\sqrt{\epsilon'(\tau)/\rho}\;d\tau,\quad T_l<T\leq0,\; v_1\leq v<v_l,\end{equation*}
where $v_1=v_l-\int_{T_l}^0\sqrt{\epsilon'(\tau)/\rho}\;d\tau$ and
$T_l<T\leq0$ is determined by the requirement of
$\lambda_2(U)>\lambda_2(U_l)$ for rarefaction waves. A simple
calculation shows that $\frac{dv}{dT}<0,$ and
$\frac{d^2v}{dT^2}\geq0 \;( T_l<T\leq0), $ where the sign of
equality holds if and only if $T=0$.

If $T>0$, then the forward rarefaction wave can not be continued
further since $\lambda_2$ is monotonically decreasing for $T>0$ (see
the second equation in (\ref{30})). In fact, the $R_2$ curve can be
continued by a forward degenerate shock curve
\begin{equation*}\label{35}S_2:\quad v=v_1-\sqrt{T\epsilon(T)/\rho}=v_l-\int_{T_l}^0\sqrt{\epsilon'(\tau)/\rho}\;d\tau-
\sqrt{T\epsilon(T)/\rho},\;\;T>0,\end{equation*} where this
degenerate shock wave satisfies the criterion (\ref{n5}) and
$\lambda_2(U_1)=s_2>\lambda_2(U_r),\; U_1=(0,v_1).$

Thus, for any state $U$ with $T>0$, we can connect the left state
$U_l$ and the right state $U$ by a forward rarefaction wave and a
degenerate shock wave.
 This kind of wave is often
called as {\it a composite wave} or {\it a multiple wave}, which
changes continuously through the rarefaction wave $R_2$ from
$(T_l,v_l)$ to $(0,v_1)$ and then jumps at the right side of $R_2$
from $(0,v_1)$ to $(T,v)$.



On the other hand, when $T<T_l$, since $\lambda_2$ is an increasing
function (see the second equation in (\ref{30})), we can not connect
the left state $U_l$ by a forward rarefaction wave $R_2$. Actually,
we can construct the following shock wave
\begin{equation}\label{37}S_2:\quad v-v_l=\sqrt{(T-T_l)[\epsilon(T)-\epsilon(T_l)]/\rho},\quad T<T_l,\end{equation}
which is the classical shock wave  satisfying (\ref{n5}) and
moreover $\lambda_2(U_l)>s_2>\lambda_2(U_r).$

Now we consider the backward elementary waves. As discussed in
Section 2, for any $T\in(T_l,T_2]$, it is possible to connect the
left state $U_l$ and a right state $U$ by a backward shock wave as
follows
\begin{equation*}\label{39}S_1:\quad v-v_l=\sqrt{(T-T_l)[\epsilon(T)-\epsilon(T_l)]},\quad T_l<T<T_2,\;\;v_l<v<v_2,\end{equation*}
where
$v_2=v_l+\sqrt{(T_2-T_l)[\epsilon(T_2)-\epsilon(T_l)]}=v_l+(T_2-T_l)\sqrt{\epsilon'(T_2)/\rho},$
in which we have made use of (\ref{38})$_1$. We find that
$\lambda_1(U_l)>s_1>\lambda_1(U_r).$ Especially, if $T=T_2$, then
this shock wave becomes a degenerate shock since
$\lambda_1(U_l)>s_1=\lambda_1(U_r).$

Furthermore, when the stress $T>T_2$, as discussed previously, we
have to continue the wave curve after $T=T_2$ by a backward
rarefaction wave:
\begin{equation*}\label{40}R_1:\quad v=v_2+\int_{T_2}^T\sqrt{\epsilon'(\tau)/\rho}\;d\tau=v_l+(T_2-T_l)\sqrt{\epsilon'(T_2)/\rho}
+\int_{T_2}^T\sqrt{\epsilon'(\tau)/\rho}\;d\tau,\;T>T_2.\end{equation*}
Thus, for any state $U=(T,v)$ with $T>T_2$, we can connect the left
state $U_l$ and the right state $U$ by a composite wave, which is
composed of an $S_1$
 wave and a $R_1$ wave.
 This
composite wave jumps at the left edge of the $R_1$ rarefaction wave
from $U_l$ to $(T_2,v_2)$ and then changes continuously through
$R_1$ from $(T_2,v_2)$ to $U$.

On the other hand, for $T<T_l$,  we can connect the left state $U_l$
and the right state $U$ by a backward rarefaction wave
\begin{equation*}\label{41}R_1:\quad v-v_l=\int_{T_l}^T\sqrt{\epsilon'(\tau)/\rho}\;d\tau,\quad T<T_l,\end{equation*}
where $T<T_l$ is determined by $\lambda_1(U)>\lambda_1(U_l)$ for a
rarefaction wave.

Case III \quad $T_l>0$.

We first consider the backward elementary waves. For $T>T_l$, we
obtain the following $R_1$ rarefaction wave
\begin{equation*}\label{42}R_1:\quad v-v_l=\int_{T_l}^T\sqrt{\epsilon'(\tau)/\rho}\;d\tau,\quad T>T_l>0,\end{equation*}
where $T>T_l>0$ is due to $\lambda_1(U)>\lambda_1(U_l)$ for
rarefaction waves.

As discussed in Section 2, for any $T\in[T_3,T_l)$, the physical
admissibility (\ref{n5}) holds and moreover
$\lambda_1(U_l)>s_1\geq\lambda_1(U)$, where $U$ is the right state.
So we have the backward shock wave
\begin{equation*}\label{44}S_1:\quad v-v_l=-\sqrt{(T-T_l)[\epsilon(T)-\epsilon(T_l)]/\rho},\quad T_3\leq T<T_l,\;v_3\leq v<v_l,\end{equation*}
where $v_3=v_l-\sqrt{(T_3-T_l)[\epsilon(T_3)-\epsilon(T_l)]/\rho}$.
Furthermore, if the stress $T$ is smaller than $T_3$,  as mentioned
in Section 2, the shock wave curve can not be continued and we have
to continue the solution by $R_1$. The $R_1$ rarefaction wave is
given by
\begin{equation*}\label{45}R_1:\quad v-v_3=\int_{T_3}^T\sqrt{\epsilon'(\tau)/\rho}\;d\tau,\quad T<T_3.\end{equation*}

Now we turn to discuss the forward elementary waves. The forward
shock wave $S_2$ is given by
\begin{equation}\label{46}S_2:\quad v-v_l=-\sqrt{(T-T_l)[\epsilon(T)-\epsilon(T_l)]/\rho},\quad T>T_l>0,\end{equation}
where $T>T_l>0$ implies the criterion (\ref{n5}) and
$\lambda_2(U_l)>s_2>\lambda_2(U)$.

For the case $0\leq T<T_l$, we can construct the following forward
rarefaction wave
\begin{equation*}\label{47}R_2:\quad v-v_l=-\int_{T_l}^T\sqrt{\epsilon'(\tau)/\rho}\;d\tau,\quad 0\leq T<T_l, v_l<v\leq v_4,\end{equation*}
where $v_4=v_l-\int_{T_l}^0\sqrt{\epsilon'(\tau)/\rho}\;d\tau$ and
the condition $0\leq T<T_l$ is derived by the requirement that
$\lambda_2(U)>\lambda_2(U_l)$ for rarefaction waves.

For $T<0$, we can not connect the left state $U_l$ and the right
state $U$ by the above $R_2$ wave. We resort to a forward shock wave
\begin{equation*}\label{48}S_2:\quad v=v_1+\sqrt{T\epsilon(T)/\rho}=v_l-\int_{T_l}^0\sqrt{\epsilon'(\tau)/\rho}\;d\tau+\sqrt
{T\epsilon(T)/\rho},\;T<0<T_l,\end{equation*} where $T<0<T_l$
implies the physical admissibility (\ref{n5}) and
$\lambda_2(U_l)>s_2>\lambda_2(U)$.

Thus, for any state $U$ with $T<0$, we can connect the states $U_l$
and  $U$ by a composite wave composed of a forward rarefaction wave
and a forward shock wave. This wave changes continuously through the
rarefaction wave $R_2$ from $U_l$ to the state $(0,v_4)$ and then
jumps at the right edge of $R_2$ from $(0,v_4)$ to $U$.

\section{Global solutions to the Riemann problem}
In this section, we construct the global Riemann solution to
(\ref{14})-(\ref{15}) according to the locations of $U_l$ and $U_r$
in the $(T,v)$ plane.

First, we place all of the wave curves $R_i=R_i(T;U_l),
S_i=S_i(T;U_l) (i=1,2)$ in the $(T,v)$ plane and find that their
distributions  vary according to $T_l<0, T_l=0$ and $T_l>0$, see
Fig. \ref{fig10}. Thus, in what follows, we divide the discussions
for the Riemann problem (\ref{14})-(\ref{15}) into three cases
$T_l<0$, $T_l>0$ and $T_l=0$.

\begin{figure}[ht!]\begin{center}
\includegraphics[width=0.32\textwidth]{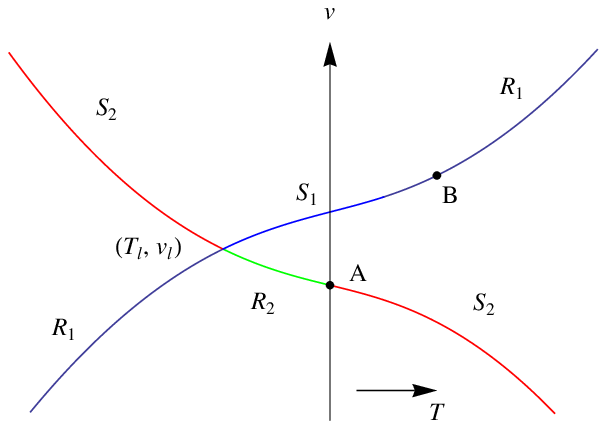}\hspace{1mm}\includegraphics[width=0.32\textwidth]{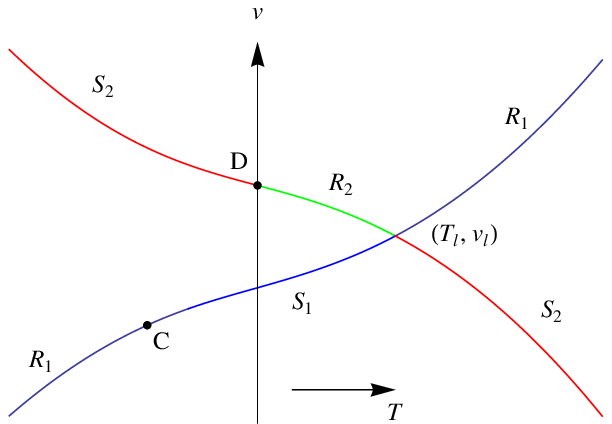}\hspace{1mm}
\includegraphics[width=0.32\textwidth]{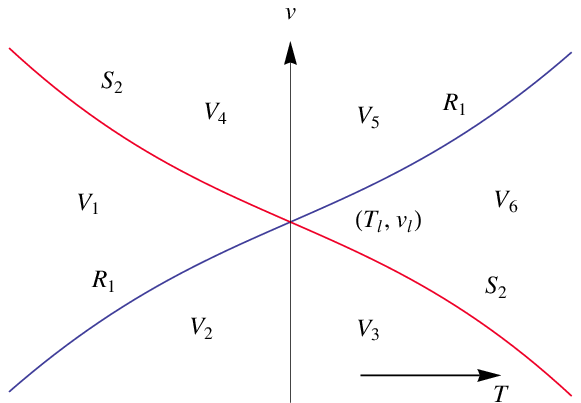}
  \renewcommand{\figurename}{Fig.} \caption{Wave curves for $T_l<0$ (left), $T_l>0$ (middle) and $T_l=0$ (right), where the coordinates of A and B are $(0,v_1)$
  and $(T_2,v_2)$, respectively,
 in which $v_1=v_l-\int_{T_l}^0\sqrt{\epsilon'(\tau)/\rho}\;d\tau$ and $v_2=v_l+(T_2-T_l)\sqrt{\epsilon'(T_2)/\rho}$, and
 the coordinates of C and D are $(T_3,v_3)$ and $(0,v_4)$, respectively,
  in which  $v_3=v_l-\sqrt{(T_3-T_l)[\epsilon(T_3)-\epsilon(T_l)]/\rho}$ and $v_4=v_l-\int_{T_l}^0\sqrt{\epsilon'(\tau)/\rho}\;d\tau$.} \label{fig10}
 \end{center}
\end{figure}

{\bf Case A} $\quad T_l<0$.

Referring to Fig. \ref{fig10}, it is well-known that the wave curves
$R_i(T;U_l)$ and $S_i(T;U_l)$ $ (i=1,2)$ have second-order contact
at the point $(T_l,v_l)$ (cf. \cite{smoller}). Moreover, we can
prove that $R_i(T;U_l)$ and $S_i(T;U_l)\, (i=1,2)$ are twice
continuously differentiable at the points A and B. As a matter of
fact, in the neighborhood of point B, the shock wave and the
rarefaction wave are given by
\begin{equation*}S_1(T;U_l): \quad
v=v_l+\sqrt{(T-T_l)[\epsilon(T)-\epsilon(T_l)]/\rho},\quad T_l<T\leq
T_2,
\end{equation*}
and \begin{equation*}R_1(T;U_l):\quad
v=v_l+(T_2-T_l)\sqrt{\epsilon'(T_2)/\rho}+\int_{T_2}^T\sqrt{\epsilon'(\tau)/\rho}\;d\tau,\quad
T>T_2,\end{equation*} respectively, where $T_2$ is determined by
(\ref{38}). First, it is easy to see that $$\lim_{T\rightarrow
T_2}\frac{\partial S_1(T;U_l)}{\partial T}=\lim_{T\rightarrow
T_2}\frac{\partial R_1(T;U_l)}{\partial
T}=\sqrt{\frac1\rho\,\epsilon'(T_2)},$$ where we have made use of
(\ref{38}). In addition, direct computation shows
$$\frac{\partial^2 R_1(T;U_l)}{\partial
T^2}=\frac{\epsilon''(T)}{2\sqrt{\rho\,\epsilon'(T)}},$$ and
$$\frac{\partial^2 S_1(T;U_l)}{\partial
T^2}=\frac{2\epsilon'(T)+(T-T_l)\epsilon''(T)}{2\sqrt{\rho\,(T-T_l)[\epsilon(T)-\epsilon(T_l)]}}-\frac{[\epsilon(T)-\epsilon(T_l)+(T-T_l)\epsilon'(T)]^2}
{4\rho^2[(T-T_l)(\epsilon(T)-\epsilon(T_l))/\rho]^{3/2}}.$$ So by
using (\ref{38}), we obtain $$\lim_{T\rightarrow
T_2}\frac{\partial^2 S_1(T;U_l)}{\partial T^2}=\lim_{T\rightarrow
T_2}\frac{\partial^2 R_1(T;U_l)}{\partial
T^2}=\frac{\epsilon''(T_2)}{2\sqrt{\rho\,\epsilon'(T_2)}}.$$ Thus,
the curves $R_1$ and $S_1$ have  second-order contact at the point
B. It is also true for the curves $R_2$ and $S_2$ at the point A.
This property for wave curves is still valid in other cases. The
proof is very similar and is omitted here.

 For convenience of discussion, we denote
\begin{equation*}\label{49}\mathcal{F}(U_l)=\left\{W_2(\tilde{U})\left|\tilde{U}\in W_1(U_l)\right.\right\},\end{equation*}
where $W_1(U)=R_1(T;U)\bigcup S_1(T;U)\bigcup R_1(T;U_B)$ is the
backward elementary wave curve issuing from $U$, in which $U_B$ is
the state at the point B, while $W_2(U)$ denotes the forward
elementary wave curve issuing from $U$. From Fig. \ref{fig10}, we
know that if the stress on the left state does not vanish, then
$W_2(U)=S_2(T;U)\bigcup R_2(T;U)\bigcup S_2(T;U)$; if the stress on
the left state equals zero, $W_2(U)=S_2(T;U)$, see Fig. \ref{fig13}.

\begin{figure}[ht!]\begin{center}
\includegraphics[width=0.65\textwidth]{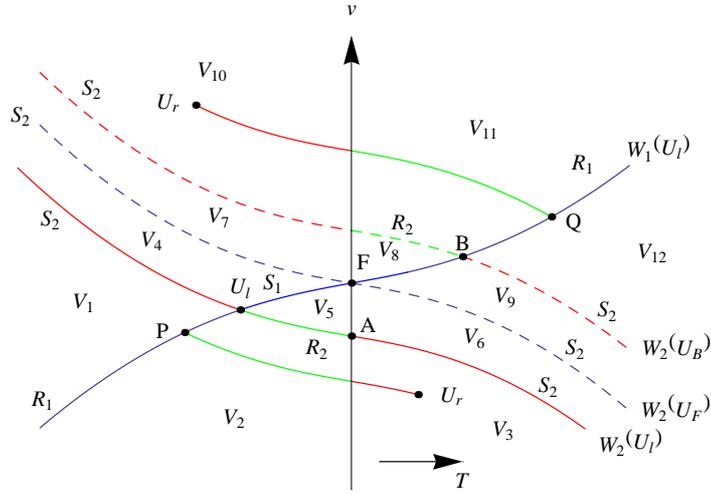} \renewcommand{\figurename}{Fig.}
 \caption{Wave curves for the case $T_l<0$, where the coordinates of points P and Q are denoted by $(\bar{T},\bar{v})$ and
 $(\hat{T},\hat{v})$, respectively and the coordinates of points B and  F are given by $(T_2,v_2)$ and $(0,v_l+\sqrt{T_l\,\epsilon(T_l)/\rho})$,
 respectively.}\label{fig13}
 \end{center}
\end{figure}

Summarizing the preceding discussions, we have
\begin{proposition}The  functions $W_i(U)\,(i=1,2)$ of the elementary wave curves
 are strictly monotone and twice continuously differentiable with respect to $T$. \end{proposition}

 We note that the wave curves vary dramatically
according  to the locations of $U_l$, see Fig. \ref{fig10}. For the
present case $T_l<0$, by the wave curves $W_1(U_l), W_2(U_l),
W_2(U_B)$, $W_2(U_F)$ and the line $T=0$, the phase plane $(T,v)$ is
divided into twelve disjoint regions V$_i$ $(i=1,2,\cdots, 12)$, see
Fig. \ref{fig13}.

 Let $U_l$ be fixed and allow $U_r$ to vary. As discussed
in Smoller \cite{smoller}, if $U_r$ lies on either $R_i$ or $S_i$
$(i=1,2)$, then the Riemann problem (\ref{14})-(\ref{15}) can be
solved as in the previous section. In order to obtain the general
solution to the Riemann problem (\ref{14})-(\ref{15}), we need to
prove that the twelve disjoint regions V$_i$ $(i=1,2,\cdots, 12)$
are covered univalently by the family of curves in
$\mathcal{F}(U_l)$. That is to say, through each point $U_r\in$
$\bigcup_{i=1}^{12}$V$_i$, there passes exactly one curve in
$\mathcal{F}(U_l)$.

Suppose $U_r\in$ V$_3$. Referring to Fig. \ref{fig13}, we have the
following equations
\begin{equation}\label{50}\quad \bar{v}=v_l+\int_{T_l}^{\bar{T}}\sqrt{\epsilon'(\tau)/\rho}\;d\tau,\end{equation}
and
\begin{equation}\label{51} v_r=\left\{\begin{array}{lll}\bar{v}-\int_{\bar{T}}^{T}\sqrt{\epsilon'(\tau)/\rho}\;d\tau,& \bar{T}<T\leq0,\vspace{2mm}\\
\bar{v}-\int_{\bar{T}}^0\sqrt{\epsilon'(\tau)/\rho}\;d\tau-\sqrt{T\epsilon(T)/\rho},&
0<T<T_r.
\end{array}\right.\end{equation}
In order to show the region V$_3$ is covered univalently by the
family of curves in $\mathcal{F}(U_l)$, it suffices to show that
$\partial T/\partial \bar{v}>0$. By (\ref{50})-(\ref{51}), we
compute
\begin{equation*}1=\sqrt{\frac{\epsilon'(\bar{T})}{\rho}}\frac{\partial\bar{T}}{\partial\bar{v}},\quad
0=\left\{\begin{array}{lll}1-\sqrt{\frac{\epsilon'(T)}{\rho}}\frac{\partial
T}
{\partial\bar{v}}+\sqrt{\frac{\epsilon'(\bar{T})}{\rho}}\frac{\partial
\bar{T}}{\partial \bar{v}},& \bar{T}<T\leq0,\vspace{2mm}\\
1+\sqrt{\frac{\epsilon'(\bar{T})}{\rho}}\,\frac{\partial\bar{T}}{\partial\bar{v}}-\frac{\epsilon(T)+T\epsilon'(T)}{2\sqrt{\rho
T\epsilon(T)}}\,\frac{\partial T}{\partial\bar{v}},&0<T<T_r.
\end{array}\right.\end{equation*}
Then, it follows that
\begin{equation*}\frac{\partial T}{\partial\bar{v}}=\left\{\begin{array}{lll}\frac{2}{\sqrt{\epsilon'(T)/\rho}}>0,& \bar{T}<T\leq0,\vspace{2mm}\\
\frac{4\sqrt{\rho
T\epsilon'(T)}}{\epsilon(T)+T\epsilon'(T)}>0,&0<T<T_r.
\end{array}\right.\end{equation*}
When $U_r$ lies in other  regions, the proof is very similar  and
the details are omitted.

Thus, the Riemann problem (\ref{14})-(\ref{15}) can be solved by
connecting $U_l$ and $\bar{U}$ by a backward (shock, or rarefaction,
or composite) wave, and then connecting $\bar{U}$ and $U_r$ by a
forward (shock, or rarefaction, or composite) wave.

Next, we list the Riemann solutions according to the locations of
$U_r$ case by case.

{\bf Case A1}\quad  If $U_r\in$ V$_1$, then the Riemann solution is
$U_l\overset{R_1}{\longrightarrow}\bar{U}\overset{S_2}{\longrightarrow}U_r,$
where  $\bar{U}=(\bar{T},\bar{v})$ is the intermediate state. The
above formula means that the state $\bar{U}$ can be connected to
$U_l$ on the right by a backward rarefaction wave and $U_r$ is
connected to $\bar{U}$ on the right by a forward shock. The symbols
below have similar meanings and we shall not explain them again
unless it is necessary.

{\bf Case A2}\quad  If $U_r\in$ V$_2$,  then the Riemann solution is
$U_l\overset{R_1}{\longrightarrow}\bar{U}\overset{R_2}{\longrightarrow}U_r.$

{\bf Case A3}\quad  If $U_r\in$ V$_3$,  then the Riemann solution is
$U_l\overset{R_1}{\longrightarrow}\bar{U}\overset{R_2}{\longrightarrow}(0,v_*)\xrightarrow[s_2=\lambda_2(0)]{S_2}U_r,$
where
$v_*=\bar{v}-\int_{\bar{T}}^0\sqrt{\epsilon'(\tau)/\rho}\;d\tau$.

{\bf Case A4}\quad   If $U_r\in$ V$_4$,  then the Riemann solution
is
$U_l\overset{S_1}{\longrightarrow}\bar{U}\overset{S_2}{\longrightarrow}U_r.$

{\bf Case A5}\quad  If $U_r\in$ V$_5$,  then the Riemann solution is
$U_l\overset{S_1}{\longrightarrow}\bar{U}\overset{R_2}{\longrightarrow}U_r.$

{\bf Case A6}\quad  If $U_r\in$ V$_6$,  then the Riemann solution is
$U_l\overset{S_1}{\longrightarrow}\bar{U}\overset{R_2}{\longrightarrow}(0,v_*)\xrightarrow[s_2=\lambda_2(0)]{S_2}U_r.$

{\bf Case A7}\quad  If $U_r\in$ V$_7$,  then the Riemann solution is
$U_l\overset{S_1}{\longrightarrow}\bar{U}\overset{R_2}{\longrightarrow}(0,v_*)\xrightarrow[s_2=\lambda_2(0)]{S_2}U_r.$
The structure of this solution is similar to that in Case A6.

{\bf Case A8}\quad  If $U_r\in$ V$_8$,  then the Riemann solution is
$U_l\overset{S_1}{\longrightarrow}\bar{U}\overset{R_2}{\longrightarrow}U_r,$
which has a similar structure to that in Case A5.

{\bf Case A9}\quad   If $U_r\in$ V$_9$,  then the Riemann solution
is
$U_l\overset{S_1}{\longrightarrow}\bar{U}\overset{S_2}{\longrightarrow}U_r,$
which is similar to Case A4.

{\bf Case A10}\quad   If $U_r\in$ V$_{10}$,  then the Riemann
solution is
$$U_l\xrightarrow[s_1=\lambda_1(T_2)]{S_1}U_B\overset{R_1}{\longrightarrow}\bar{U}\overset{R_2}{\longrightarrow}(0,v_*)
\xrightarrow[s_2=\lambda_2(0)]{S_2}U_r,$$ where
$U_B=(T_2,v_2)=(T_2,v_l+(T_2-T_l)\sqrt{\epsilon'(T_2)/\rho})$, in
which $T_2$ is given by (\ref{38}).



{\bf Case A11}\quad   If $U_r\in$ V$_{11}$,  then the Riemann
solution is
$U_l\xrightarrow[s_1=\lambda_1(T_2)]{S_1}U_B\overset{R_1}{\longrightarrow}\bar{U}\overset{R_2}{\longrightarrow}U_r.$

{\bf Case A12}\quad   If $U_r\in$ V$_{12}$,  then the Riemann
solution is
$U_l\xrightarrow[s_1=\lambda_1(T_2)]{S_1}U_B\overset{R_1}{\longrightarrow}\bar{U}\overset{S_2}{\longrightarrow}U_r.$

So we have solved the Riemann problem (\ref{14})-(\ref{15}) for the
case $T_l<0$.

{\bf Case B} $\quad T_l>0$.

 For this case, we are able to construct the
solutions to the Riemann problem (\ref{14})-(\ref{15}) by a very
similar method to that in Case A. We first draw all the wave curves
in the $(T,v)$ phase plane and also split the entire plane into
twelve disjoint regions V$_i (i=1,2,\cdots,12)$ by the wave curves
$W_1(U_l), W_2(U_l), W_2(U_E), W_2(U_C)$ and the line $T=0$, see
Fig. \ref{fig16}.

\begin{figure}[ht!]\begin{center}
\includegraphics[width=0.65\textwidth]{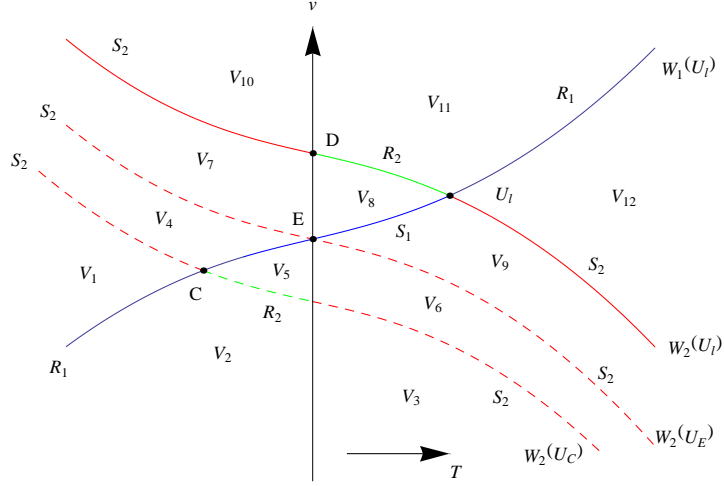} \renewcommand{\figurename}{Fig.}
 \caption{Riemann solutions for the case $T_l>0$, where the coordinates for C, D and E are $(T_3,v_3), (0,v_4)$ and
 $(0,v_l-\sqrt{T_l\,\epsilon(T_l)/\rho}\,)$, respectively.}\label{fig16}
 \end{center}
\end{figure}

As before, we can prove that each region V$_i$ is covered
univalently by the family of curves $\mathcal{F}(U_l)$. So if $U_r$
lies in each wave curves, then the Riemann solution can be derived
easily as in the previous section; while, if $U_r$ lies in one of
regions V$_i (i=1,2,\cdots,12)$, we can construct the corresponding
Riemann solution as follows.

{\bf Case B1}\quad  If  $U_r\in$ V$_1$,  then the Riemann solution
is
$U_l\xrightarrow[s_1=\lambda_1(T_3)]{S_1}U_C\overset{R_1}{\longrightarrow}\bar{U}\overset{S_2}{\longrightarrow}U_r,$
where
$U_C=(T_3,v_3)=(T_3,v_l-\sqrt{(T_3-T_l)[\epsilon(T_3)-\epsilon(T_l)]/\rho}\,)$,
in which $T_3$ is determined by (\ref{38}).

{\bf Case B2}\quad  If  $U_r\in$ V$_2$, then the Riemann solution is
$U_l\xrightarrow[s_1=\lambda_1(T_3)]{S_1}U_C\overset{R_1}{\longrightarrow}\bar{U}\overset{R_2}{\longrightarrow}U_r.$

{\bf Case B3}\quad  If  $U_r\in$ V$_3$, then the Riemann solution is
$U_l\xrightarrow[s_1=\lambda_1(T_3)]{S_1}U_C\overset{R_1}{\longrightarrow}\bar{U}\overset{R_2}{\longrightarrow}(0,v_*)
\xrightarrow[s_2=\lambda_2(0)]{S_2}U_r,$ where
$v_*=\bar{v}-\int_{\bar{T}}^0\sqrt{\epsilon'(\tau)/\rho}\;d\tau$.

{\bf Case B4}\quad  If  $U_r\in$ V$_4$, then the Riemann solution is
$U_l\overset{S_1}{\longrightarrow}\bar{U}\overset{S_2}{\longrightarrow}U_r.$

{\bf Case B5}\quad  If  $U_r\in$ V$_5$, then the Riemann solution is
$U_l\overset{S_1}{\longrightarrow}\bar{U}\overset{R_2}{\longrightarrow}U_r.$

{\bf Case B6}\quad  If  $U_r\in$ V$_6$, then the Riemann solution is
$U_l\overset{S_1}{\longrightarrow}\bar{U}\overset{R_2}{\longrightarrow}(0,v_*)\xrightarrow[s_2=\lambda_2(0)]{S_2}U_r.$

{\bf Case B7}\quad  If  $U_r\in$ V$_7$, then the Riemann solution is
$U_l\overset{S_1}{\longrightarrow}\bar{U}\overset{R_2}{\longrightarrow}(0,v_*)\xrightarrow[s_2=\lambda_2(0)]{S_2}U_r.$

{\bf Case B8}\quad  If  $U_r\in$ V$_8$, then the Riemann solution is
$U_l\overset{S_1}{\longrightarrow}\bar{U}\overset{R_2}{\longrightarrow}U_r.$

{\bf Case B9}\quad  If  $U_r\in$ V$_9$, then the Riemann solution is
$U_l\overset{S_1}{\longrightarrow}\bar{U}\overset{S_2}{\longrightarrow}U_r.$

{\bf Case B10}\quad  If  $U_r\in$ V$_{10}$,  then the Riemann
solution is
$U_l\overset{R_1}{\longrightarrow}\bar{U}\overset{R_2}{\longrightarrow}(0,v_*)\xrightarrow[s_2=\lambda_2(0)]{S_2}U_r.$

{\bf Case B11}\quad  If $U_r\in$ V$_{11}$, then the Riemann solution
is
$U_l\overset{R_1}{\longrightarrow}\bar{U}\overset{R_2}{\longrightarrow}U_r.$

{\bf Case B12}\quad  If  $U_r\in$ V$_{12}$, then the Riemann
solution is
$U_l\overset{R_1}{\longrightarrow}\bar{U}\overset{S_2}{\longrightarrow}U_r.$

So we have finished the construction of Riemann solutions for the
Case B.

{\bf Case C} $\quad T_l=0$.

In this case, we place the wave curves $R_1(T;U_l)$ and $S_2(T;U_l)$
in the $(T,v)$ plane, where $R_1(T;U_l)$ and $S_2(T;U_l)$ are given
by (\ref{31}) and (\ref{32}), respectively. It is seen that the
$(T,v)$ plane is divided into six disjoint regions V$_i\,
(i=1,2,\cdots,6)$, see Fig. \ref{fig10}. Similarly, one can verify
that each region V$_i$ can be covered univalently by the family of
curves in $\mathcal{F}(U_l)$. So the Riemann problem
(\ref{14})-(\ref{15}) can be solved as in the following forms.

{\bf Case C1}\quad  If  $U_r\in$ V$_1$, then the Riemann solution is
$U_l\overset{R_1}{\longrightarrow}\bar{U}\overset{S_2}{\longrightarrow}U_r$.

{\bf Case C2}\quad  If  $U_r\in$ V$_2$, then the Riemann solution is
$U_l\overset{R_1}{\longrightarrow}\bar{U}\overset{R_2}{\longrightarrow}U_r$.

{\bf Case C3}\quad  If  $U_r\in$ V$_3$, then the Riemann solution is
$U_l\overset{R_1}{\longrightarrow}\bar{U}\overset{R_2}{\longrightarrow}(0,v_*)\xrightarrow[s_2=\lambda_2(0)]{S_2}U_r$.

{\bf Case C4}\quad  If  $U_r\in$ V$_4$, then the Riemann solution is
$U_l\overset{R_1}{\longrightarrow}\bar{U}\overset{R_2}{\longrightarrow}(0,v_*)\xrightarrow[s_2=\lambda_2(0)]{S_2}U_r$.

{\bf Case C5}\quad  If  $U_r\in$ V$_5$, then the Riemann solution is
$U_l\overset{R_1}{\longrightarrow}\bar{U}\overset{R_2}{\longrightarrow}U_r$.

{\bf Case C6}\quad  If  $U_r\in$ V$_6$, then the Riemann solution is
$U_l\overset{R_1}{\longrightarrow}\bar{U}\overset{S_2}{\longrightarrow}U_r$.

So we have obtained the Riemann solutions completely for the Case C.

Thus, we have obtained the globally unique piecewise smooth
solutions to the Riemnn problem (\ref{14})-(\ref{15}). Summarizing
the above discussions, we have
\begin{theorem}There exists a unique piecewise smooth solution to the Riemnn problem (\ref{14})-(\ref{15}) with any given
initial Riemann data. These solutions are composed of constant
states, backward (forward) rarefaction wave, backward (forward)
shock wave and composite wave, which is a combination of rarefaction
wave and shock wave.
\end{theorem}

\section{A physically realizable case}
In this section, we consider a case which can be realized in an
experimental setting. Consider an infinitely-long circular elastic
rod which is bonded by a thin rigid ring around its middle section
(say, at $x=0$). Axial forces are applied to generate a stress $T_l$
for the part of $x<0$ and a stress $T_r$ for the part of $x>0$.
Then, the rigid ring is released, and stress waves will be
generated. Mathematically, this corresponds to the Riemann problem
(\ref{14})-(\ref{15}) with the velocities $v_l=v_r=0$. Although it
is a special case contained in the general results given the
previous section, we provide more details for the Riemann solutions
due to the physical relevance.

 First we make some preliminary observations. When $T_l<0$, by referring to Fig. \ref{fig13}, we denote the horizontal coordinates
 of the intersections for curves $W_2(U_F), W_2(U_B)$ and the axes $v=0$ by $T_*$ and $T_{**}$,
 respectively. By Eq. (\ref{32}) and Eq. (\ref{46}) and the formulas for the states at points F and B,  we find that $T_*$
 and $T_{**}$ are determined by
 \begin{equation*}\label{52}\sqrt{T_l\,\epsilon(T_l)/\rho}=\sqrt{T_*\,\epsilon(T_*)/\rho},\end{equation*}
and
\begin{equation*}\label{53}(T_2-T_l)\sqrt{\epsilon'(T_2)/\rho}=\sqrt{(T_{**}-T_2)[\epsilon(T_{**})-\epsilon(T_2)]/\rho},\end{equation*}
respectively. Moreover, we have $T_{**}>T_*>0$ and $T_*=-T_l$ by
(\ref{16}).

When $T_l>0$, we represent the horizontal coordinates of the
intersections for curves $W_2(U_E), W_2(U_C)$ and the axes $v=0$ by
$\hat{T}_*$ and $\hat{T}_{**}$, respectively. Similarly, by Eqs.
(\ref{32}) and (\ref{37}), we find that $\hat{T}_*$ and
$\hat{T}_{**}$ satisfy
$\sqrt{T_l\,\epsilon(T_l)/\rho}=\sqrt{\hat{T}_*\epsilon(\hat{T}_*)/\rho}$
and\begin{equation*}\label{55}\sqrt{(T_3-T_l)[\epsilon(T_3)-\epsilon(T_l)]/\rho}=\sqrt{(\hat{T}_{**}-T_3)[\epsilon(\hat{T}_{**})-\epsilon(T_3)]/\rho}
,\end{equation*} respectively. Furthermore, we  have
$\hat{T}_{**}<\hat{T}_*<0$ and $\hat{T}_*=-T_l$ by (\ref{16}).

For the convenience of comparison with the linear case, we solve the
the Riemann problem (\ref{14})-(\ref{15}) with the linearized
strain-stress function being $\epsilon=(\alpha+\beta)T$. By the
corresponding Riemann invariants and the method of characteristics,
it is straightforward to obtain the following solution
\begin{equation*}\label{56}(T,v)=\left\{\begin{array}{lll}(T_l,v_l)&x<\lambda_1t,\\
\left(\frac12\left[T_r+T_l+\sqrt{\frac{\rho}{\alpha+\beta}}(v_r-v_l)\right],\frac12\left[v_r+v_l+\sqrt{\frac{\alpha+\beta}{\rho}}(T_r-T_l)
\right]\right)&\lambda_1t<x<\lambda_2t,\\
(T_r,v_r)&x>\lambda_2t,\end{array}\right.\end{equation*} where
$\lambda_1=-\frac{1}{\sqrt{\rho(\alpha+\beta)}}$  and
$\lambda_2=\frac{1}{\sqrt{\rho(\alpha+\beta)}}.$
 Especially,
when $v_l=v_r=0$, the above solution reduces into
\begin{equation}\label{57}(T,v)=\left\{\begin{array}{lll}(T_l,0)&x<\lambda_1t=-\frac{1}{\sqrt{\rho(\alpha+\beta)}}\;t,\\
\left(\frac12\left(T_r+T_l\right),\frac12\sqrt{\frac{\alpha+\beta}{\rho}}(T_r-T_l)
\right)&\lambda_1t<x<\lambda_2t=\frac{1}{\sqrt{\rho(\alpha+\beta)}}\;t,\\
(T_r,0)&\frac{1}{\sqrt{\rho(\alpha+\beta)}}\;t=\lambda_2t<x,\end{array}\right.\end{equation}
which will be useful in the following discussions.

Now we are ready to explore the Riemann solutions in detail for the
physical case $v_l=v_r=0$. The discussions are divided into the
twelve cases as follows. The stress profiles corresponding to the
different locations of the Riemann initial data are depicted in
detail in Figs. \ref{fig17} and \ref{fig18}.

\begin{figure}[th!]\begin{center}
\includegraphics[width=0.4\textwidth]{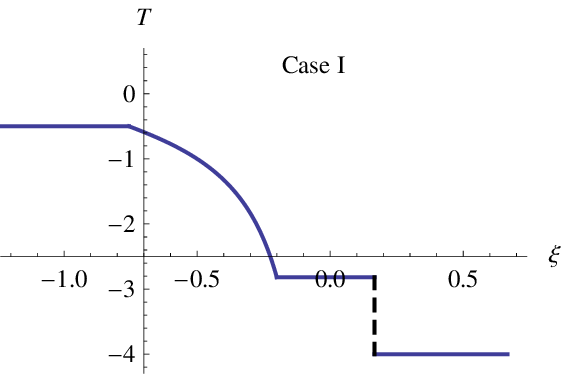}\hspace{2mm}\includegraphics[width=0.4\textwidth]{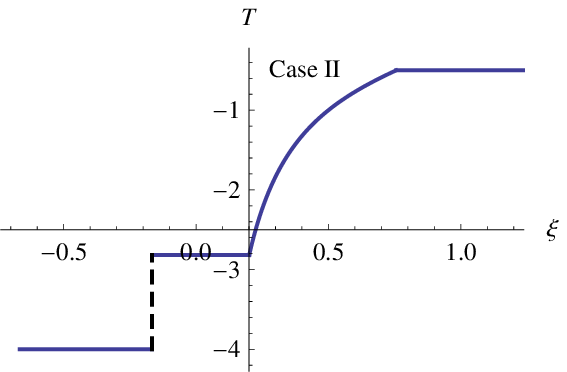}
\includegraphics[width=0.4\textwidth]{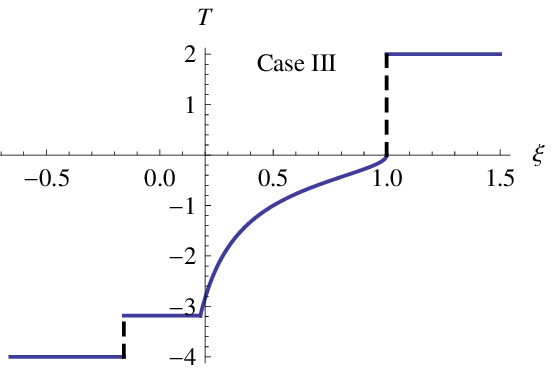}\hspace{2mm}\includegraphics[width=0.4\textwidth]{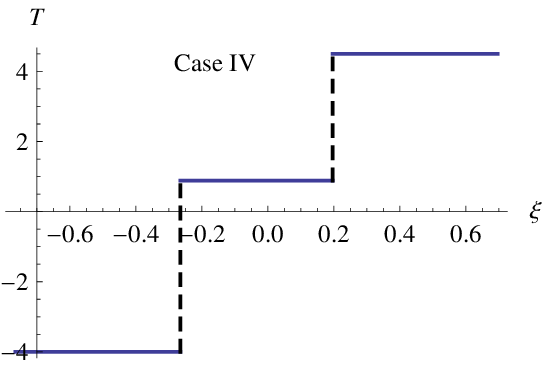}
\includegraphics[width=0.4\textwidth]{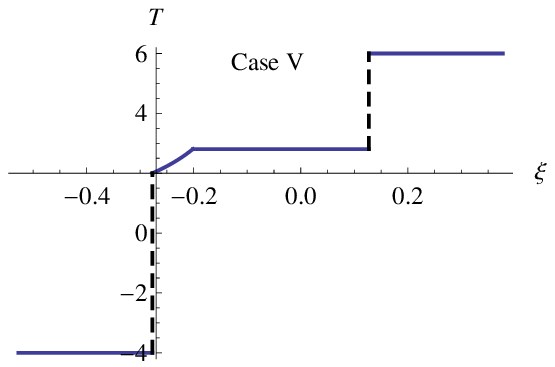}\hspace{2mm}\includegraphics[width=0.4\textwidth]{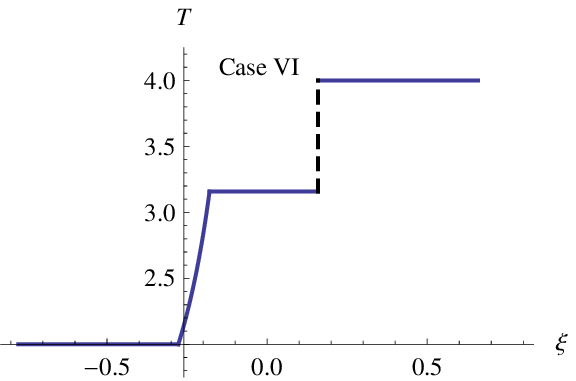}
  \renewcommand{\figurename}{Fig.} \caption{Cauchy stress profiles for different locations of Riemann initial data $U_l$.}\label{fig17}
 \end{center}
\end{figure}

{\bf Case I} \quad  If $T_r<T_l<0$, the Riemann solution is
$U_l\overset{R_1}{\longrightarrow}\bar{U}\overset{S_2}{\longrightarrow}U_r.$
More precisely, the solution formula is given by
\begin{equation*}\label{58}(T,v)=\left\{\begin{array}{lll}(T_l,0)&x<\xi_1t,\\
(\hat{T}(\xi),\hat{v}(\xi))&\xi_1 t\leq x\leq
\xi_2t,\\(\bar{T},\bar{v})&\xi_2t<x<s_2t,\\(T_r,0)&s_2t<x,\end{array}\right.\end{equation*}
where $\xi_1=-1/\sqrt{\rho\epsilon'(T_l)},
\xi_2=-1/\sqrt{\rho\epsilon'(\bar{T})}$ and
$(\hat{T}(\xi),\hat{v}(\xi))$ is determined by
\begin{equation}\label{59}\xi=-\frac{1}{\sqrt{\rho\epsilon'(\hat{T}(\xi))}},\quad \hat{v}(\xi)=\int_{T_l}^{\hat{T}}\sqrt{\frac{\epsilon'(\tau)}{\rho}}
\,d\tau.\end{equation} In addition, by the argument in Section 2, we
calculate the intermediate state $(\bar{T},\bar{v})$ via
\begin{equation}\label{60}\bar{v}=\int_{T_l}^{\bar{T}}\sqrt{\frac{\epsilon'(\tau)}{\rho}}\,d\tau,\quad \bar{v}+\sqrt{(\bar{T}-T_r)[\epsilon(\bar{T})-
\epsilon(T_r)]/\rho}=0,\end{equation} and the speed of forward shock
is
\begin{equation}\label{61}s_2=\frac{1}{\sqrt{\rho\frac{\epsilon(T_r)-\epsilon(\bar{T})}{T_r-\bar{T}}}}.\end{equation}

\begin{figure}[th!]\begin{center}
\includegraphics[width=0.4\textwidth]{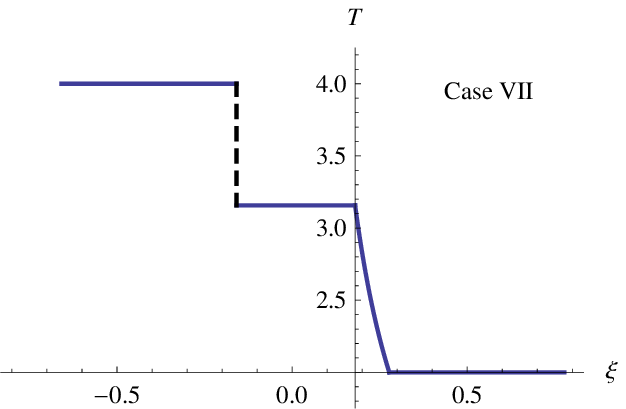}\hspace{2mm}\includegraphics[width=0.4\textwidth]{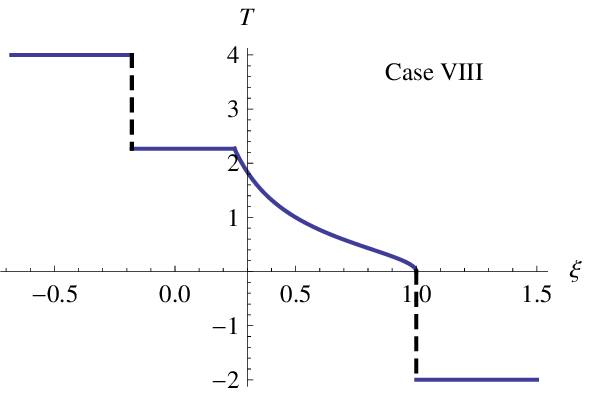}
\includegraphics[width=0.4\textwidth]{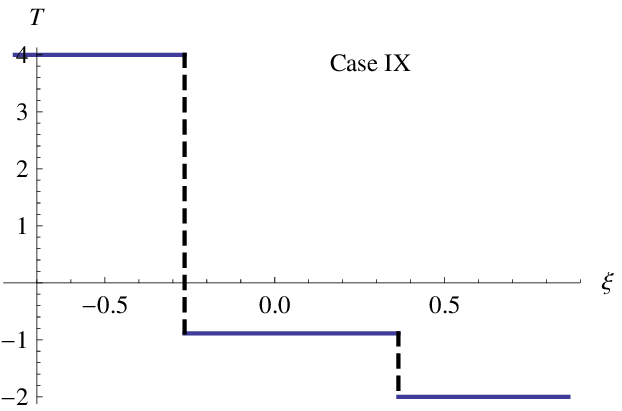}\hspace{2mm}\includegraphics[width=0.4\textwidth]{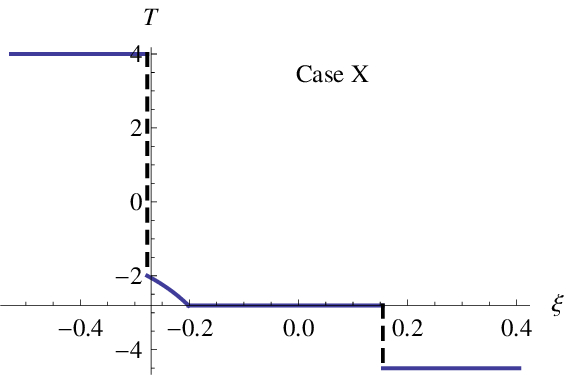}
\includegraphics[width=0.4\textwidth]{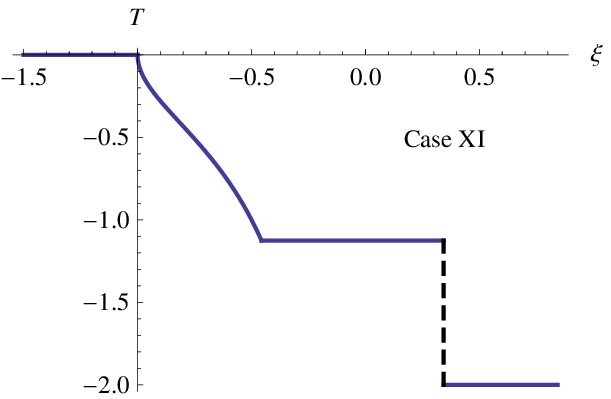}\hspace{2mm}\includegraphics[width=0.4\textwidth]{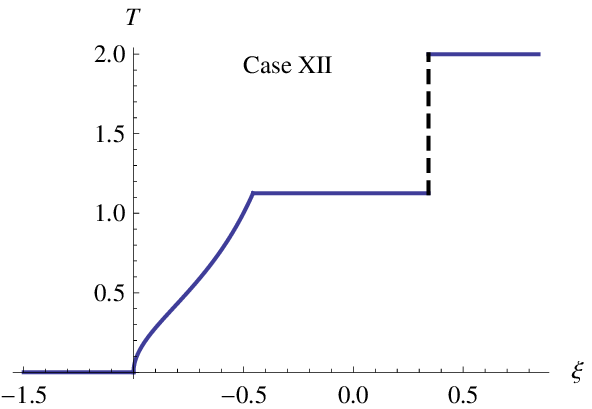}
  \renewcommand{\figurename}{Fig.} \caption{Cauchy stress profiles for different locations of Riemann initial data $U_l$
 (Continued).}\label{fig18}
 \end{center}
\end{figure}

It is easy to observe that if the strain-stress relation is linear,
then there is no rarefaction wave (cf. (\ref{59})). Moreover, it
follows from the first equation in (\ref{59}) and Eq. (\ref{61})
that $\xi_1=\xi_2=\lambda_1, s_2=\lambda_2$.  By (\ref{60}), we
further have
$$(\bar{T},\bar{v})=\left(\frac12\left(T_r+T_l\right),\frac12\sqrt{\frac{\alpha+\beta}{\rho}}(T_r-T_l)
\right).$$ So it is noted that all these results are consistent with
the solution (\ref{57}).

{\bf Case II}  \quad If $T_l<T_r\leq0$, the Riemann solution is
$U_l\overset{S_1}{\longrightarrow}\bar{U}\overset{R_2}{\longrightarrow}U_r.$

{\bf Case III}  \quad If $0<T_r<T_*$, the Riemann solution is
$U_l\overset{S_1}{\longrightarrow}\bar{U}\overset{R_2}{\longrightarrow}(0,v_*)\xrightarrow[s_2=\lambda_2(0)]{S_2}U_r.$

{\bf Case IV}  \quad If $T_*\leq T_r\leq T_{**}$, the Riemann
solution is
$U_l\overset{S_1}{\longrightarrow}\bar{U}\overset{S_2}{\longrightarrow}U_r,$

{\bf Case V}  \quad If $T_{**}<T_r$, the Riemann solution is
$U_l\xrightarrow[s_1=\lambda_1(T_2)]{S_1}U_B\overset{R_1}{\longrightarrow}\bar{U}\overset{S_2}{\longrightarrow}U_r.$

{\bf Case VI} \quad  If $T_r>T_l>0$, the Riemann solution is
$U_l\overset{R_1}{\longrightarrow}\bar{U}\overset{S_2}{\longrightarrow}U_r.$

{\bf Case VII} \quad  If $0\leq T_r<T_l$, the Riemann solution is
$U_l\overset{S_1}{\longrightarrow}\bar{U}\overset{R_2}{\longrightarrow}U_r.$

{\bf Case VIII} \quad  If $-T_l=\hat{T}_*<T_r<0$, the Riemann
solution is
$U_l\overset{S_1}{\longrightarrow}\bar{U}\overset{R_2}{\longrightarrow}(0,v_*)\xrightarrow[s_2=\lambda_2(0)]{S_2}U_r.$

{\bf Case IX} \quad  If $\hat{T}_{**}\leq T_r\leq\hat{T}_*=-T_l$,
the Riemann solution is
$U_l\overset{S_1}{\longrightarrow}\bar{U}\overset{S_2}{\longrightarrow}U_r.$

{\bf Case X} \quad  If $T_r<\hat{T}_{**}$, the Riemann solution is
$U_l\xrightarrow[s_1=\lambda_1(T_3)]{S_1}U_C\overset{R_1}{\longrightarrow}\bar{U}\overset{S_2}{\longrightarrow}U_r.$

{\bf Case XI} \quad  If $T_r<T_l=0$, the Riemann solution is
$U_l\overset{R_1}{\longrightarrow}\bar{U}\overset{S_2}{\longrightarrow}U_r$.

{\bf Case XII} \quad  If $T_r>T_l=0$, the Riemann solution is
$U_l\overset{R_1}{\longrightarrow}\bar{U}\overset{S_2}{\longrightarrow}U_r$.

In summary, there are in total twelve wave patterns depending on the
initial stresses, while for a classical linearly elastic material
there is only one wave pattern.
 Thus, if this case is realized in an experiment, by measuring the wave patterns one
  can determine whether the material is a classical one or the one which belongs
  to the sub-class of new elastic bodies defined through equation (\ref{13}).

\vspace{2mm}
\section*{Acknowledgments}
This work was completed when the first author (S.-J. Huang) was
visiting Professor Huai-Dong Cao at Lehigh University. The first
author would like to thank Professor Cao and Mathematics Department
in Lehigh University for great hospitality.

 \end{document}